\documentclass[article]{elsarticle}
\usepackage{framed}
\addcontentsline{toc}{section}{Acknowledgement}
\usepackage{geometry}
\usepackage{lineno,hyperref}
\usepackage{amssymb,amsmath,amsthm}
\usepackage{mathrsfs}
\numberwithin{equation}{section}
\usepackage{lineno,hyperref}
\usepackage{graphicx}
\usepackage{subfigure}
\usepackage{float}
\usepackage{color}

\newtheorem{theorem}{Theorem}[section]
\newtheorem{lemma}{Lemma}[section]

\newtheorem{thm}{Theorem}[section]

\newtheorem{remark}[thm]{Remark}
\newtheorem{definition}{Definition}[section]

\def\<{\left\langle} \def\>{\right\rangle}
\def\\left({\left\left(} \def\\right){\right\right)}

\begin{document}
	
	\begin{frontmatter}

	\title{{Spreading speeds of a nonlocal diffusion model with free
boundaries in the time almost periodic
media}
			 \tnoteref{mytitlenote}}

		
\author{{ Chengcheng Cheng}\corref{mycorrespondingauthor}}
\cortext[mycorrespondingauthor]{Corresponding author}
\ead{chengchengcheng@amss.ac.cn}
\author{{ Rong Yuan}}
\ead{ryuan@bnu.edu.cn}

		\address{\scriptsize{ School of Mathematical Sciences, 
			Laboratory 
			of 
			Mathematics 
			and 
			Complex Systems, Ministry of Education, \\ Beijing Normal University, Beijing 
			100875, China} }

	
 \begin{abstract}
 { In this paper, we mainly investigate the spreading dynamics of a nonlocal diffusion KPP model with free boundaries which is firstly explored in time almost periodic media.
  As the spreading occurs, the long-run dynamics are obtained. Especially, when the threshold condition 
for the kernel function is satisfied, applying the novel positive time almost periodic function, we accurately express the unique asymptotic spreading speed of the free boundary problem.}
\end{abstract} 
	
\end{frontmatter}
  { \section{Introduction}\label{s1}
 Nonlocal diffusion equations take advantage of modeling the long-distance movements of species and the dense dispersal of populations~\cite{murray2002mathematical, 2023-hutson2003evolution, 2023-kao2010random}.
Considering the environmental change, seasonal succession and resource distribution, people often choose periodic models to describe environmental parameters~\cite{2023-rawal2012criteria, 2023-souganidis2019, 2023-bao2020propagation, 2023-sun2022limiting}. However,
 natural fluctuations are difficult to be periodic.  Almost periodicity is more likely to accurately describe natural recurrent changes and the almost periodic model can provide more biological insights to understand the invasion of populations and the propagation of epidemic diseases~\cite{2023-diagana2007population, 2023-zhao2017dynamical}.
It is well known that the spreading speed is an important indicator describing the scale of infectious disease transmission.  The spreading speeds of nonlocal diffusion equations in the nonhomogeneous media have attracted increasingly more interest and attention in recent years~\cite{2023-shen2010spreading, 2023-liang2020spreading}.
However, the spreading speed of the nonlocal KPP model with free boundaries in the time almost periodic environment has not been considered, which is a worthwhile problem to study. 
 
 In this paper, we mainly aim to investigate the long-time propagation dynamics and spreading speed of  the following almost periodic nonlocal diffusion model
\begin{equation}\label{eq1-1}
	\left\{\begin{array}{ll}
		u_{ t}=d \int_{g\left(t\right)}^{h\left(t\right)} 
		\kappa\left(x-y\right) u\left(t, y\right) \mathrm{d} y-d 
		u+u f(t,x,u), & t>0,~
		g\left(t\right)<x<h\left(t\right), \vspace{0.1cm} \\
		u\left(t, h(t)\right)=u(t,g(t))=0, & t > 0,  \vspace{0.1cm} \\
		h^{\prime}\left(t\right)= \mu
		\int_{g\left(t\right)}^{h\left(t\right)} 
		\int_{h\left(t\right)}^{\infty} \kappa\left(x-y\right) u\left(t, 
		x\right) \mathrm{d} y\mathrm{d} x, 
		& t 
		> 0, \vspace{0.1cm} \\
		g^{\prime}\left(t\right)=-\mu
		\int_{g\left(t\right)}^{h\left(t\right)} 
		\int_{-\infty}^{g\left(t\right)} \kappa\left(x-y\right) u\left(t, 
		x\right) \mathrm{d} y\mathrm{d} 
		x, & 
		t >0, \vspace{0.1cm}  \\
		u\left(0, x\right)=u_{ 0}\left(x\right),& x\in[-h_{0},h_0], \vspace{0.1cm}
		\\ h\left(0\right)=h_0, g\left(0\right)=-h_{0}, 
	\end{array}\right.
\end{equation}
where the initial function $u_0(x)$ satisfies 
\begin{equation}\label{eq1-2}
	u_{0}(x) \in C(\left[-h_{0}, h_{0}\right]),~ u_{0}\left(-h_{0}\right)=u_{0}\left(h_{0}\right)=0,~ u_{0}(x)>0,~ x\in\left(-h_{0}, h_{0}\right).
\end{equation}
Here $u(t,x)$ represents the population density at time $t$ on location $x$. We assume that the density of species is $0$ out of $[g(t), h(t)]$, where  $g(t)$ and $h(t)$ are the leftward and rightward free boundaries to be determined later, respectively.
The dispersal kernel $\kappa(x-y)$ represents the probability distribution that jumps from location $y$ to $x$, then $\kappa(x-y) u(t, y)$ is the rate at which the species arrive at $x$ from position $y$. 
And $\int_{g(t)}^{h(t)} \kappa (x-y) u(t, y) \mathrm{d} y$ is the rate at which the species reach $x$ from any other positions. 
Correspondingly, $\int_{-\infty}^{\infty} \kappa(x-y) u(t, x) \mathrm{d} y$ is the rate at which the 
species leave $x$ to go to any other positions. Let $d$ be a positive constant as the nonlocal diffusion coefficient, then
$
d\int_{g(t)}^{h(t)} \kappa(x-y) u(t, y) \mathrm{d} y-\int_{-\infty}^{\infty} \kappa(x-y) u(t, x) \mathrm{d} y
$
can be seen as the nonlocal diffusion term.
Moreover, 
similar to \cite{cao2019}, suppose that the expanding rate of moving boundary is proportional to the outward
flux at the boundary, that is,
\begin{equation}\label{gg}
	g^{\prime}\left(t\right)=-\mu
		\int_{g\left(t\right)}^{h\left(t\right)} 
		\int_{-\infty}^{g\left(t\right)} \kappa\left(x-y\right) u\left(t, 
		x\right) \mathrm{d} y\mathrm{d} 
		x, ~ 
		t >0, 
\end{equation}		
		and
\begin{equation}\label{hh}	
		h^{\prime}\left(t\right)= \mu
		\int_{g\left(t\right)}^{h\left(t\right)} 
		\int_{h\left(t\right)}^{\infty} \kappa\left(x-y\right) u\left(t, 
		x\right) \mathrm{d} y\mathrm{d} x, 
		~ t 
		> 0, 
\end{equation}
where $\mu>0$ describes the expanding capability.

Moreover, we assume that the kernel function $\kappa(x): \mathbb R\rightarrow \mathbb R$ satisfies the following properties.
\newline  $(\textbf{H1})$ (1) $\kappa(\cdot) \in C^1(\mathbb{R},[0, \infty))$ is symmetric, $\kappa(0)>0,$ and $\int_{\mathbb{R}} \kappa(x) d x=1$.
\newline (2) There are positive constants $\alpha$ and $x_0$ such that $\kappa(x) \leq e^{-\alpha|x|}$ and $|\nabla \kappa| \leq e^{-\alpha|x|}$ for $|x| \geq x_0$.

The growth function $f(t,x,u): \mathbb R\times \mathbb R\times \mathbb R^+\longrightarrow \mathbb R$ satisfies that for $D\subset \mathbb R$ is smooth bounded or $D=\mathbb R,$
\newline   $(\textbf{H2})$ $f(t, x, u)$ is almost periodic in $t$ uniformly for $x \in \bar{D}$ and $u$ in bounded sets of $\mathbb{R}$. When $D=\mathbb{R}$,  $f(t, x, u)$ is also almost periodic in $x$ uniformly for $t \in \mathbb{R}$ and $u$ in bounded sets.
\newline   $(\textbf{H2}^*)$   (1) $f(t, x, u)$ is $C^1$ in $u$;   
$f(t, x, u)$ and $f_u(t, x, u)$ are uniformly continuous in $(t, x, u) \in \mathbb{R} \times \bar{D} \times E$ for any bounded set $E \subset \mathbb{R} ; f(t, x, u)$ is uniformly bounded in $(t, x) \in \mathbb{R} \times D$ for $u$ in bounded sets. 
\newline (2) There is $M\gg 1$ such that $ f(t, x, u)<0$ for  $(t, x) \in \mathbb{R} \times \bar{D}$ and $u \geq M$;  $f_u(t, x, u)<0$ for $(t, x, u) \in \mathbb{R} \times \bar{D} \times[0, \infty).$


Let $a(t,x)=f(t,x,0)$ satisfy
\newline  $(\textbf{H3})$ $a(t, x)$ is bounded and uniformly continuous in $(t, x) \in \mathbb{R} \times \bar{D}$, and is almost periodic in $t$ uniformly for $x \in \bar{D}$.
When $D=\mathbb R$, $a(t, x)$ is also almost periodic in $x$ uniformly for $t \in\mathbb R$.


Additionally, for the principal Lyapunov exponent as the threshold condition in studying the spreading and vanishing of $(\ref{eq1-1})$, we assume that 
  \newline   $(\textbf{H4})$ There exists a $L^*>0$ such that $\inf\limits_{L\geq L^*}\lambda_{\mathscr {PL}}(a, L)>0$, where $\lambda_{\mathscr {PL}}(a, L)$ is the principal Lyapunov exponent of the following equation
\begin{equation}
\begin{cases}
u_t=d\int_{-L}^{L}\kappa (y-x) u(t,y ){\rm d} y- d u(t,x)+a (t,x) u (t,x), & x\in(-L,L), \vspace{0.1cm}
\\u(t,-L)=u(t,L)=0.
\end{cases}
\end{equation}

In this article, unless otherwise specified, $u(t, x ; u_0)$ always denotes the solution of $(\ref{eq1-1})$ with $u(0, x; u_0)=u_0(x).$
Sometimes, we write $u(t, x; u_0, a)$ to emphasize the solution dependent on $a(t,x)$.
Assume that $(\mathbf{H1})-(\mathbf{H2}^*)$ hold, for any given $u_0$ satisfying $(\ref{eq1-2})$, according to Theorem~1.1 \cite{cao2019}, the problem $(\ref{eq1-1})$ admits a unique solution 
$
u(t, x; u_0, g, h)\text{ for all }t>0
$
 with
  \begin{equation}
 u(0, x; u_0)=u_0(x) \text { and  }  0\leq u(t, x; u_0)\leq \max\left\{ |u_0|_{L^\infty}, M\right\}.
 \end{equation} 
where $M$ comes from $(\textbf{H2}^*)$. 
Moreover, $g(t; u_0, h_0)$ is strictly decreasing in $t>0$, and $h(t; u_0,h_0)$ is strictly increasing in $t>0.$ 
Thus, there are $g_\infty\in[-\infty, 0)$ and $h_\infty\in(0,\infty]$ such that $\lim\limits_{t\rightarrow\infty}g(t; u_0, h_0)=g_\infty\text{ and } \lim\limits_{t\rightarrow\infty}h(t; u_0, h_0)=h_\infty.$
Simultaneously,  according to Theorem~1.2~\cite{2023-chengalmost1},  we give the spreading-vanishing dichotomy regimes for $(\ref{eq1-1})$.
\begin{theorem}[Spreading-vanishing dichotomy]\label{th1-2}
Assume that $(\mathbf{H1})-(\mathbf{H4})$ hold. For any $h_0>0$ and $u_0$ satisfying $(\ref{eq1-2})$, the either of the followings must hold:\vspace{0.1cm}
\newline$(1)$ Vanishing:  $h_\infty-g_\infty \leq 2L^*$, and $\lim\limits_{t\rightarrow\infty} u(t,x;u_0)=0$ uniformly for $x\in[g_\infty,h_\infty];$
\newline $(2)$ Spreading: $h_\infty-g_\infty=\infty$, and $\lim\limits_{t\rightarrow\infty}\left[u(t,x;u_0)-u^*(t,x)\right]=0$ uniformly for $x$ in any compact subset of $\mathbb R,$ where $u^*(t,x)$ is the unique time almost periodic solution of the following equation
\begin{equation}\label{eq1-3}
u_t=d\int_{\mathbb R} \kappa (y-x) u(t, y){\rm d}y-d u(t,x)+uf(t, x, u), ~x\in \mathbb R.
\end{equation}
\end{theorem}

\section{Main results}
In this section,  we devote ourselves to exploring the spreading speed of the leftward and rightward front for $(\ref{eq1-1}). $

 For the homogeneous nonlocal diffusion model with free boundaries, Du, Li and Zhou investigated the spreading speed of the moving boundaries of the nonlocal diffusion model by applying the semi-wave and traveling wave solutions in~\cite{du2021semi}. 
 However,  for the difficulties caused by the time almost periodicity, the semi-wave solution of the time almost periodic nonlocal diffusion equations could be a rough nut to crack.
Thus, the subtle methods developed by them could not be directly used in the problem $(\ref{eq1-1}).$  We had to look for and develop new methods to solve this problem.

Firstly, consider an assumption about the kernel function which is called satisfying the ``thin-tailed" condition if the following condition holds.
\begin{equation}\label{thin}
 (\textbf{H})\qquad\qquad\qquad\int_{-\infty}^{0}\int_{0}^{\infty}
\kappa(x-y){\rm d}y{\rm d}x=\int_{0}^{\infty}x\kappa(x){\rm d}x<\infty.
\end{equation}
In the case of $f(t,x,u)\equiv f(t,u).$ Considering the following free boundary problem
\begin{equation}\label{eq4-1}
	\left\{\begin{array}{ll}
		u_{ t}=d \int_{g(t)}^{h\left(t\right)} 
		\kappa\left(x-y\right) u\left(t, y\right) \mathrm{d} y-d~
		u+u f(t,u), & t>0,~
		g(t) <x<h\left(t\right), \vspace{0.1cm}  \\
		u\left(t, h(t)\right)=0, u(t,g(t))=0, & t > 0, \vspace{0.1cm} \\
		h^{\prime}\left(t\right)= \mu
		\int_{g(t)}^{h\left(t\right)} 
		\int_{h\left(t\right)}^{\infty} \kappa\left(x-y\right) u\left(t, 
		x\right) \mathrm{d} y\mathrm{d} x, 
		& t 
		> 0,  \vspace{0.1cm} \\
		g^{\prime}\left(t\right)= -\mu
		\int_{g(t)}^{h\left(t\right)} 
		\int_{-\infty}^{0} \kappa\left(x-y\right) u\left(t, 
		x\right) \mathrm{d} y\mathrm{d} x, 
		& t 
		> 0,  \vspace{0.1cm} \\
		u\left(0, x\right)=u_{ 0}\left(x\right),~
		h\left(0\right)=h_{0},~ g(0)=-h_0, & x\in[-h_0, h_0],
	\end{array}\right.
\end{equation}
and the fixed boundary problem 
\begin{equation}\label{eq4-3}
u_t=d\int_{\mathbb R}\kappa(x-y) u(t,y){\rm d}y -d ~ u(t,x)+uf(t, u),  ~x\in\mathbb R.
\end{equation}
by Theorem~4.1~\cite{2023-shen1998convergence} and Theorem~1.3~\cite{2023-onyido2021non2}, there is a unique positive  time almost periodic solution $\hat u^*(t,x)$  of $(\ref{eq4-3})$  with 
\begin{equation}\label{eq04-3}
\hat u^*(t,x)\equiv\hat u^*(t),
\end{equation} 
 where $\hat u^*(t)$ is the positive time almost periodic solution of the following  problem
\begin{equation}
\setlength{\abovedisplayskip}{0pt} 
 \setlength{\belowdisplayskip}{0pt}
u_t=uf(t,u).
\end{equation}

Since $\hat u^*(t)$ is almost periodic in $t$, by Theorem 3.1~\cite{2023-fink2006almost}, there is $\hat u^*$ such that
\begin{equation}\label{eq4-5}
\hat u^*=\lim\limits_{T\rightarrow\infty}\frac{1}{T}\int_{0}^{T}\hat u^{*}(t){\rm d}t.
\end{equation} 
Denote
$$c^*:=\mu\int_{-\infty}^{0}\int_{0}^{\infty}\kappa(x-y) \hat u^* {\rm d}y{\rm d}x.
$$
Assume that $\textbf{(H)}$ holds, one can see that $c^*$ is well defined.

Now we mainly give an explicit estimate of the asymptotic spreading speed,
that is,  
\begin{theorem}[Finite spreading]\label{th1-4}
Assume that $(\mathbf{H1})-(\mathbf{H4})$ hold. When the spreading occurs, if $(\mathbf{H})$ holds, the spreading speeds of the leftward front and the rightward front satisfy that  
\begin{equation}\label{eq1-7}
\lim\limits_{t\rightarrow \infty} \dfrac{h(t)}{t}=\lim\limits_{t\rightarrow\infty}\dfrac {-g(t)}{t}=c^*.
\end{equation}
\end{theorem}
\begin{theorem}\label{th02-2}
Under the conditions of Theorem~\ref{th1-4}, 
for any $\varepsilon\in(0, c^*)$, it follows
\begin{equation}\label{eq-s1}
\lim _{t \rightarrow \infty} \max _{|x| \leq\left(c^*-\varepsilon\right) t}\left|u(t, x)-\hat u^*(t)\right|=0.
\end{equation}
\end{theorem}

\begin{theorem}[Accelerated spreading]\label{th1-04}
Assume that $(\mathbf{H1})-(\mathbf{H4})$ hold. When the spreading occurs, if $(\mathbf{H})$ does not hold, it follows that
\begin{equation}\label{eq1-8}
\lim\limits_{t\rightarrow \infty} \dfrac{h(t)}{t}=\lim\limits_{t\rightarrow\infty}\dfrac {-g(t)}{t}=\infty.
\end{equation}
\end{theorem}

\section{Proofs}
\begin{proof}[The proof of Theorem~\ref{th1-4}]
 For simplify, denote $u(t,x)=u(t,x;u_0)$. As $\hat{u}(t, x):=u(t,-x)$ satisfies $(\ref{eq4-1})$ with free boundaries 
 $$\setlength{\abovedisplayskip}{0pt} 
 \setlength{\belowdisplayskip}{0pt}
 x=\hat{h}(t):=-g(t),~ x=\hat{g}(t):=-h(t)$$ 
 and initial function $\hat{u}_0(x):=u_0(-x)$, we only need to prove the case of $h(t)$. The spreading speed for $g(t)$ can be directly obtained.

Now we complete the proof of this theorem in two steps.

${Step~1:}$ Firstly, we prove that $\setlength{\abovedisplayskip}{0pt} 
 \setlength{\belowdisplayskip}{0pt}
 \limsup\limits_{t\rightarrow\infty}\dfrac{h(t)}{t}\leq c^*.$

For $0<\epsilon\ll 1$, let $\overline u_\epsilon(t)$ and $\underline u_\epsilon(t)$ be the positive time almost periodic solutions of the following problem
\begin{equation}\label{eq4-6}
u_t=u(f(t,u)+\epsilon)
\end{equation}
and 
\begin{equation}\label{eq4-7}
u_t=u(f(t,u)-\epsilon),
\end{equation}
respectively. Thus,
\begin{equation}\label{eq4-8}
\lim\limits_{\epsilon \rightarrow 0} \underline u_\epsilon(t)=\lim\limits_{\epsilon\rightarrow 0}\overline u_\epsilon(t)=\hat u^*(t)
\end{equation}
uniformly in $t\in\mathbb R.$
 Applying a comparison argument, we can see that 
 \begin{equation}\label{l0}
 \underline  u_\epsilon(t) \leq \hat u^*(t) \leq  \overline u_\epsilon(t).
 \end{equation}
 Thus, according to Theorem~\ref{th1-2}, there exists $T>0$ such that $u(t+T, x)\leq \overline u_\epsilon(t+T), \text { for } t\geq 0, ~x\in[-h(t+T),h(t+T)].$
Let $$\tilde u(t,x)=u(t+T,x), ~\tilde h(t)=h(t+T),~\tilde g(t)=g(t+T),$$ then $\tilde u(t,x)$ satisfies
\begin{equation}
	\left\{\begin{array}{ll}
		\tilde u_{ t}=d \int_{\tilde g(t)}^{\tilde h\left(t\right)} 
		\kappa\left(x-y\right) \tilde u\left(t, y\right) \mathrm{d} y-d 
		\tilde u+\tilde u f(t+T,  \tilde u), & t>0,~
		\tilde g\left(t \right)<x<\tilde h\left(t\right),  \vspace{0.1cm} \\
		 \tilde u(0,x)=u(T,x), & t > 0,~ \tilde g(0)<x< \tilde h(0),  \vspace{0.1cm} \\
		\tilde h^{\prime}\left(t\right)= \mu
		\int_{\tilde g(t)}^{\tilde h\left(t\right)} 
		\int_{\tilde h\left(t\right)}^{\infty} \kappa\left(x-y\right) \tilde u\left(t, 
		x\right) \mathrm{d} y\mathrm{d} x, 
		& t 
		> 0,  \vspace{0.1cm} \\
		\tilde g^{\prime}\left(t\right)=- \mu
		\int_{\tilde g(t)}^{\tilde h\left(t\right)} 
		\int_{-\infty}^{\tilde g(t)} \kappa\left(x-y\right) \tilde u\left(t, 
		x\right) \mathrm{d} y\mathrm{d} x, 
		& t 
		> 0,  \vspace{0.1cm} \\
		\tilde u(t, \tilde h(t))=0, ~\tilde u(t,\tilde g(t))=0, & t>0.
	\end{array}\right.
\end{equation}
Let $u^*(t)$ be the solution to the following initial problem
\begin{equation}
\begin{cases}
u_t=u(f(t,u)+\epsilon)), &t>0,\vspace{0.1cm} \\
u(T)=\max\{\overline u_\epsilon(T), \|\tilde u(0,\cdot)\|_{\infty} \},
\end{cases}
\end{equation}
then $u^*(t)\geq \overline u_\epsilon (t), \text {for } t\geq T,$ and $$\lim\limits_{t\rightarrow\infty}(u^*(t)-\overline u_\epsilon(t))=0. $$
Moreover, by Theorem~\ref{th1-2}, combining $(\ref{eq04-3})$ with $(\ref{eq4-8})$, we have  $$\setlength{\abovedisplayskip}{0pt} 
 \setlength{\belowdisplayskip}{0pt} \tilde u(t,x)\leq u^*(t+T), \text { for } t\geq 0, -\tilde h(t)<x<\tilde h(t),$$
then $$\setlength{\abovedisplayskip}{0pt} 
 \setlength{\belowdisplayskip}{0pt} \tilde u(t,x)\leq \dfrac{\overline u_{\epsilon}(t+T)}{1-\epsilon}, \text { for } t\gg 1, -\tilde h(t)<x<\tilde h(t).$$
Let $u_\epsilon(t,x)$ be the unique positive almost periodic solution of the following problem
\begin{equation}\label{eq4-11}
u_t=d\int_{\mathbb R}\kappa(x-y) u(t,y){\rm d}y -d  u(t,x)+u(f(t, u)+\epsilon), ~x\in\mathbb R,
\end{equation}
 by  the uniqueness of almost periodic solution for $(\ref{eq4-11})$ explained in Theorem~1.4~\cite{2023-onyido2021non2}, we have 
$$\setlength{\abovedisplayskip}{0pt} 
 \setlength{\belowdisplayskip}{0pt}
u_\epsilon(t,x) \equiv  u_\epsilon(t)= \overline u_\epsilon(t).$$
Thus, there is $T^*>T\gg1$ such that $$u_\epsilon(t,x)\geq (1-\epsilon)\overline u_\epsilon(t), \text{ for }t>T^*, -\tilde h(t)<x<\tilde h(t) .$$
Therefore, $$\setlength{\abovedisplayskip}{0pt} 
 \setlength{\belowdisplayskip}{0pt} \tilde u(t,x)\leq \dfrac{\overline u_\epsilon(t+T)}{1-\epsilon}\leq \dfrac{u_\epsilon(t+T, x)}{(1-\epsilon)^2}, \text { for }t>T^*, -\tilde h(t)<x<\tilde h(t).$$
For fixed $L_0>0$, take
\begin{equation}
\begin{aligned}
&\overline  h(t)=(1-\epsilon)^{-2}\mu\int_{0}^{t} \int_{-\infty}^{0}\int_{0}^{\infty} \kappa(x-y) u_\epsilon(s, x){\rm d}y{\rm d}x{\rm d}s
+\tilde h(T^*)+L_0, \text { for }t\geq0,
\\& \overline u(t,x)={(1-\epsilon)^{-2}}{u_\epsilon(t,x) }, \text { for } t\geq0,~x\in\mathbb R.
\end{aligned}
\end{equation}
According to assumption $(\textbf{H2}^*)$  for $f(t,u)$, direct computations give
\begin{equation}
\begin{aligned}
\overline u_t&=(1-\epsilon)^{-2} u_{\epsilon, t}
=(1-\epsilon)^{-2} \left[d\int_{\mathbb R}\kappa (x-y)u_{\epsilon}(t,y){\rm d}y-d ~u_\epsilon (t,x)  +u_\epsilon f(t,u_\epsilon) \right]
\\&\geq  d\int_{-\overline h(t)}^{\overline h(t)}\kappa (x-y) \overline u(t,y){\rm d}y-d ~\overline u (t,x)  +\overline u  f(t, u_\epsilon) 
\\&\geq d\int_{-\overline h(t) }^{\overline h(t)}\kappa (x-y) \overline u(t,y){\rm d}y-d ~\overline u (t,x)  +\overline u  f(t, \overline u).
\end{aligned}
\end{equation}
Meanwhile,
\begin{equation}
\begin{aligned}
\\\overline h^{\prime}(t)&= \mu\int_{-\infty}^{0}\int_{0}^{\infty} \kappa(x-y) \overline u(t,x) {\rm d}y{\rm d}x
\\   & = \mu\int_{-\infty}^{\overline h(t)}\int_{\overline h(t)}^{\infty} \kappa(x-y) \overline u(t,x) {\rm d}y{\rm d}x
\\&\geq  \mu\int_{-\overline h(t) }^{\overline h(t)}\int_{\overline h(t)}^{\infty} \kappa(x-y) \overline u(t,x) {\rm d}y{\rm d}x.
\end{aligned}
\end{equation}
Moreover, $$\overline u(T,x)={(1-\epsilon)^{-2}}{u_\epsilon(T, x) }\geq \tilde u(0, x), \text{ for } x\in (-\tilde h(T), \tilde h(T))$$
with  $\overline h(T^*)\geq\tilde h(T^*)$ and
$\overline u(t, \overline h(t))\geq 0, \text { for } t\geq 0.$
Applying the Comparison principle, one can see that 
\begin{equation}\label{l1}
\overline u(t+T, x)\geq \tilde u(t,x),~ \overline h(t)\geq \tilde h(t), \text { for }t\geq T^*,  -\tilde h(t)<x< \tilde h(t),
\end{equation}
which implies that 
$$  \begin{aligned}
\limsup\limits_{t\rightarrow\infty}\dfrac{h(t)}{t}&=\limsup\limits_{t\rightarrow \infty}\dfrac{\tilde h(t-T)}{t}
\leq \limsup\limits_{t\rightarrow \infty} \dfrac{\overline h(t-T)}{t}
\\&= \lim\limits_{t\rightarrow\infty} \dfrac{(1-\epsilon)^{-2}\mu\int_{0}^{t-T} \int_{-\infty}^{0}\int_{0}^{\infty} \kappa(x-y) u_\epsilon(s, x){\rm d}y{\rm d}x{\rm d}s
+\tilde h(T^*)+L_0}{t} 
\\&= \lim\limits_{t\rightarrow\infty} \dfrac{(1-\epsilon)^{-2}\mu\int_{0}^{t-T} \int_{-\infty}^{0}\int_{0}^{\infty} \kappa(x-y) u_\epsilon(s, x){\rm d}y{\rm d}x{\rm d}s
}{t}.
\end{aligned}$$
Since $u_\epsilon(t,x)\rightarrow \hat u^*(t,x)$ as $\epsilon\rightarrow 0$ uniformly in $t\in \mathbb R$ and $x$ in any compact subsets of $\mathbb R$,  given ${\rm (\mathbf{H})}$ and $(\ref{eq4-5})$, it follows that 
\begin{equation}\label{eq4-12}
\begin{aligned}
\limsup\limits_{t\rightarrow\infty}\dfrac{h(t)}{t}&\leq \limsup\limits_{t\rightarrow \infty} \dfrac{\overline h(t-T)}{t}
= \lim\limits_{t\rightarrow\infty} \dfrac{\mu\int_{0}^{t-T} \int_{-\infty}^{0}\int_{ 0}^{\infty} \kappa(x-y) \hat u^*(s, x){\rm d}y{\rm d}x{\rm d}s
}{t}
\\&=\lim\limits_{t\rightarrow\infty} \dfrac{\mu\int_{0}^{t-T} \int_{-\infty}^{0}\int_{ 0}^{\infty} \kappa(x-y) \hat u^*(s){\rm d}y{\rm d}x{\rm d}s
}{t}
\\&=\lim\limits_{t\rightarrow\infty} \dfrac{\mu\int_{0}^{t} \int_{-\infty}^{0}\int_{ 0 }^{\infty} \kappa(x-y) \hat u^*(s){\rm d}y{\rm d}x{\rm d}s
}{t}
\\&=\mu \int_{-\infty}^{0}\int_{ 0 }^{\infty} \kappa(x-y) \hat u^* { \rm d}y{\rm d}x
=c^*.
\end{aligned}
\end{equation} 

${Step~2:}$ We will prove $\liminf\limits_{t\rightarrow \infty}\dfrac {h(t)}{t} \geq c^*.$

Note that there exists a unique positive almost periodic solution $u_*(t, x)\equiv u_*(t) $ of the problem
$$
u_t=d \int_{\mathbb R} 
		\kappa\left(x-y\right)  u\left(t, y\right) \mathrm{d} y-d~
		 u+ u f(t+T, u),~x\in\mathbb R.
$$
By Theorem~$\ref{th1-2}$, it follows that
$$
\lim _{t \rightarrow \infty}\left[\tilde{u}(t, x)-u_*(t, x)\right]=0
$$
locally uniformly in $x\in \mathbb R$. In view of $(\ref{eq4-7})$, there is $T_*\gg1$ such that
$$
u_*(t,x) \geq \underline{u}_\epsilon(t+T), \text { for } t\geq T_*,
$$
locally uniformly for $x\in \mathbb R$, which implies that $$ \setlength{\abovedisplayskip}{0pt} 
 \setlength{\belowdisplayskip}{0pt}
\liminf\limits_{t \rightarrow \infty}\left[u_*(t,x)-\underline{u}_\epsilon(t+T)\right] \geq 0$$
locally uniformly for $x\in \mathbb R$.
Thus, $$
\liminf\limits_{t \rightarrow \infty}\left[\tilde u(t,x)-\underline{u}_\epsilon(t+T)\right] \geq 0$$
locally uniformly for $x\in \mathbb R$.

For fixed $L\gg h(T)$, let $u_{\epsilon^-}(t,x) $ be the unique time almost periodic solution of the following problem
\begin{equation}\label{eq4-14}
\begin{cases}
u_t=d\int_{-L}^{L}\kappa(x-y) u(t,y){\rm d}y -d~u(t,x)+u(f(t, u)-\epsilon),  &t>0, ~x\in(-L,L), \vspace{0.1cm}
\\ u(t,x)=0,&t>0,~|x|\geq L.
\end{cases}
\end{equation}
Applying a comparison argument,  
$u_{\epsilon^-}(t,x) \leq  \underline u_{\epsilon}(t),$
 uniformly in $t>0$ and  $x$ in any compact subsets of $\mathbb R.$
 And $u_{\epsilon^-}(t,x) \rightarrow \underline u_{\epsilon}(t) \text { as } L\rightarrow \infty$ uniformly in $t>0$  and $x$ in any compact subsets of $\mathbb R.$

Take $\tilde{T}> T$  such that $h(\tilde T)>L$. 
Define
$$
\begin{aligned}
&\underline h(t)=(1-\epsilon)^2 (1-2\epsilon)\mu \int_{\tilde{T}+T}^t  \int_{-\infty}^{0}\int_{0}^{\infty} \kappa (x-y) u_{\epsilon^-}(s, x) {\rm d} y{\rm d}x {\rm d}s + \tilde{h}(\tilde{T}),
~~\underline u(t, x)=(1-\epsilon)^2 u_{\epsilon^-} (t, x),
\end{aligned}
$$
for $t \geq \tilde{T}+T, ~x\in\mathbb R.
$ Then
$$
\tilde{u}(t, 0) \geq \underline u(t+T, 0), \text { for } t \geq \tilde{T}
\text {and }
\tilde{u}(\tilde{T}, x) \geq \underline u(\tilde{T}+T, x), \text { for }-\underline h(\tilde T+T) \leq x \leq \underline h(\tilde{T}+T).
$$ 
Thus,  we have
$$
\begin{aligned}
\underline h^{\prime}(t)&= (1-2\epsilon)\mu \int_{-\infty}^{0}\int_{0}^{\infty} \kappa (x-y) \underline u(t, x) {\rm d} y{\rm d}x
\\&=(1-2\epsilon)  \mu \int_{-\underline h(t)}^{\underline h(t) }\int_{\underline h(t) }^{\infty} \kappa (x-y) \underline u(t, x) {\rm d} y{\rm d}x
+\mu\int_{-\infty}^{-\underline h(t)}\int_{\underline h(t)}^{\infty} \kappa (x-y) \underline u(t, x) {\rm d} y{\rm d}x
\\&=\mu \int_{-\underline h(t)}^{\underline h(t) }\int_{\underline h(t) }^{\infty} \kappa (x-y) \underline u(t, x) {\rm d} y{\rm d}x+
\mu\int_{-\infty}^{-\underline h(t)}\int_{\underline h(t)}^{\infty} \kappa (x-y) \underline u(t, x) {\rm d} y{\rm d}x
\\&-2\epsilon \mu \int_{-\underline h(t)}^{\underline h(t) }\int_{\underline h(t) }^{\infty} \kappa (x-y) \underline u(t, x) {\rm d} y{\rm d}x
\\&:=\mu \int_{-\underline h(t)}^{\underline h(t) }\int_{\underline h(t) }^{\infty} \kappa (x-y) \underline u(t, x) {\rm d} y{\rm d}x+\mathscr{I}(\epsilon,\mu,h(t)) \text { for }  t \geq \tilde{T}+T.
\end{aligned}
$$

Now we only need to prove $\mathscr{I} (\epsilon,\mu,h(t))\leq 0.$
In fact, given $(\textbf{H})$, one can see that if  $t\gg\tilde T+T$, it follows that
$$
\mu\int_{-\infty}^{-\underline h(t)}\int_{\underline h(t)}^{\infty} \kappa (x-y) \underline u(t, x) {\rm d} y{\rm d}x
\leq 2\epsilon \mu \int_{-\underline h(t)}^{\underline h(t) }\int_{\underline h(t) }^{\infty} \kappa (x-y) \underline u(t, x) {\rm d} y{\rm d}x,$$
then $\mathscr{I} (\epsilon,\mu,h(t))\leq 0,$
which implies  $$\underline h^{\prime}(t)\leq \mu \int_{-\underline h(t)}^{\underline h(t) }\int_{\underline h(t) }^{\infty} \kappa (x-y) \underline u(t, x) {\rm d} y{\rm d}x.$$

Moreover, in view of $(\ref{eq4-14})$, it follows that $$
\underline u(t, \pm \underline h(t))\rightarrow 0, \text { as } t\gg \tilde T+T.$$ 
According to ${\rm (\mathbf{H1})}$,  
we have $$\setlength{\abovedisplayskip}{0pt} 
 \setlength{\belowdisplayskip}{0pt} \int_{\underline h(t)}^{\infty}\kappa(x-y) {\rm d}y+\int_{-\infty}^{-\underline h(t) }\kappa(x-y) {\rm d}y\leq \epsilon, \text { for } t \gg \tilde{T}+T.$$ 
Direct calculations show 
\[
\begin{aligned}
\underline u_t&=(1-\epsilon)^{2} u_{\epsilon^-, t}
=(1-\epsilon)^{2} \left[d\int_{\mathbb R}\kappa (x-y)u_{\epsilon^-}(t,y){\rm d}y-d ~ u_{\epsilon^-} (t,x) +u_{\epsilon^-}  (f(t,u_{\epsilon^-})-\epsilon) \right]
\\&\leq  d\int_{-\underline h(t)}^{\underline h(t)}\kappa (x-y) \underline u(t,y){\rm d}y-d~\underline u (t,x)  +\underline u  f(t, u_{\epsilon^-}) +\int_{\underline h(t)}^{\infty}\kappa(x-y) \underline u(t,y){\rm d}y
+\int_{-\infty}^{-\underline h(t)}\kappa(x-y) \underline u(t,y){\rm d}y-\epsilon ~\underline u(t,x)
\\&\leq d\int_{-\underline h(t)}^{\underline h(t)}\kappa (x-y) \underline u(t,y){\rm d}y-d~\underline u (t,x)  +\underline u  f(t, \underline u).
\end{aligned}
\]
Hence,  using the Comparison principle to conclude that
\begin{equation}\label{l2}
\begin{aligned}
&\underline h(t+T) \leq \tilde{h}(t), \text { for } t \gg \tilde{T},
\text { and } \underline u(t+T, x) \leq \tilde{u}(t, x), \text { for } t \gg \tilde{T}, -\underline h(t+T)<x<\underline h(t+T).
\end{aligned}
\end{equation}
It follows that
$$
\begin{aligned}
\liminf\limits_{t \rightarrow \infty} \frac{h(t)}{t}&=\liminf\limits_{t \rightarrow \infty} \frac{\tilde{h}(t-T)}{t}  \geq\liminf\limits_{t \rightarrow \infty} \frac{\underline h(t)}{t} \\
& =\liminf\limits_{t \rightarrow \infty} \frac{(1-\epsilon)^2(1-2\epsilon)\mu  \int_{\tilde{T}+T}^t  \int_{-\infty}^{0 }\int_{0}^{\infty} \kappa (x-y) u_{\epsilon^-}(s, x) {\rm d} y{\rm d}x {\rm d}s + \tilde{h}(\tilde{T})}{t} \\
& =(1-\epsilon)^2(1-2\epsilon)  \lim_{t \rightarrow \infty} \frac{ \mu \int_{0}^t  \int_{-\infty}^{0 }\int_{0}^{\infty} \kappa (x-y) u_{\epsilon^-}(s, x) {\rm d} y{\rm d}x {\rm d}s }{t}.
\end{aligned}
$$
For $h_\infty=\infty,$ there is $t^*>\tilde T+T>0$ such that $h(t)-h_\infty<\epsilon,\text { as } t\geq t^*.$
Let $L = h(t^*)$, note that $u_{\epsilon^-}(t,x) \rightarrow \hat {u}^{*}(t,x)$ as $\epsilon \rightarrow 0$ uniformly in $t >0$ and $x\in \mathbb (-L, L)$. Thus,
\begin{equation}\label{eq4-15}
\begin{aligned}
\liminf\limits_{t \rightarrow \infty} \frac{h(t)}{t} &\geq \lim\limits _{t \rightarrow \infty} \dfrac{ \mu \int_{0}^{t} \int_{-\infty}^{0 }\int_{0}^{\infty} \kappa (x-y)\hat u^* (s, x) {\rm d} y{\rm d}x{\rm d}s}{t}
=\lim\limits_{t \rightarrow \infty} \dfrac{ \mu \int_{0}^{t}\int_{-\infty}^{-h(t^*)} \int_{0}^{\infty }  \kappa (x-y)\hat u^* (s) {\rm d} y{\rm d}x{\rm d}s}{t}
\\&+\lim\limits_{t \rightarrow \infty} \dfrac{ \mu \int_{0}^{t}\int_{-h(t^*)}^{0} \int_{0}^{\infty} \kappa (x-y)\hat u^* (s) {\rm d} y{\rm d}x{\rm d}s}{t}
=\lim\limits _{t \rightarrow \infty} \dfrac{ \mu \int_{0}^{t} \int_{-\infty}^{0 }\int_{0}^{\infty} \kappa (x-y)\hat u^* (s) {\rm d} y{\rm d}x{\rm d}s}{t}
\\&= { \mu  \int_{-\infty}^{0 }\int_{0}^{\infty} \kappa (x-y)\hat u ^*{\rm d} y{\rm d}x}
=c^*.
\end{aligned}
\end{equation}
Thus, combining $(\ref{eq4-12})$ and $(\ref{eq4-15})$,  one can see that
\begin{equation}\label{cc}
\lim\limits_{t \rightarrow \infty} \frac{h(t)}{t}=c^* .
\end{equation}
\end{proof}

\begin{proof}[The proof of Theorem~\ref{th02-2}]
Now we intend to show $(\ref{eq-s1}).$

According to $(\ref{l1})$ in {Step~1} of the proof, for any given small $\sigma>0$, there is $T_\sigma>T$ such that
$$
u(t, x) \leq(1-\sigma)^{-2} u_\sigma(t, x), \text { for } t \geq T_\sigma, -\tilde h(t-T) \leq x \leq \tilde{h}(t-T)
$$
where $u_\sigma(t,x)$ is the unique almost periodic solution of $(\ref{eq4-11})$ with $\epsilon$ replaced by $\sigma$ and {$\tilde{h}(t)=h(t+T).$}

Moreover, by $(\ref{l2})$ in Step~2 of the proof, there are positive constants $\underline {T}_\sigma>T$ and $h_\sigma>0$ such that
$$
u(t, x) \geq(1-\sigma)^2 u_{\sigma^{-}}(t, x),\text { for } t \geq \underline {T}_\sigma,  -\underline h(t)\leq  x \leq \underline h(t)
$$
where
$$
\underline h(t)=(1-\sigma)^2(1-2\sigma) \mu \int_{\underline {T}_\sigma}^{t}  \int_{-\infty}^{0}\int_{0}^{\infty} \kappa (x-y) u_{\sigma^-}(s, x) {\rm d} y{\rm d}x {\rm d}s+h_\sigma,
$$
and $u_{\sigma^{-}}$ is the unique almost periodic solution of $(\ref{eq4-14})$ with $\epsilon$ replaced by $\sigma$.

As
$$
\begin{aligned}
&\lim _{\sigma \rightarrow 0}(1-\sigma)^{-2} \mu\int_{-\infty}^{0}\int_{0}^{\infty}\kappa(x-y)u_{\sigma}(t,x){\rm d}y{\rm d}x
=\lim _{\sigma \rightarrow 0}(1-\sigma)^2 (1-2\sigma)\mu\int_{-\infty}^{0}\int_{0}^{\infty}\kappa(x-y)u_{\sigma^-}(t,x){\rm d}y{\rm d}x
=c^*
\end{aligned}
$$
uniformly for $t \geq 0$.
For any $\varepsilon\in(0,c^*)$, there are much small $\sigma_\varepsilon \in(0, \varepsilon)$ and $T_{\varepsilon}>0$ such that 
$$
\left|(1-\sigma_\varepsilon)^{-2}   \mu \int_{0}^{t} \int_{-\infty}^{0}\int_{0}^{\infty}\kappa(x-y)u_{\sigma_\varepsilon}(s, x){\rm d}y{\rm d}x{\rm d}s-c^* t\right|<\frac{\varepsilon}{2} t
$$
and
$$
\left|(1-\sigma_\varepsilon)^2 (1-2\sigma_\varepsilon)   \mu    \int_{0}^{t} \int_{-\infty}^{0}\int_{0}^{\infty}\kappa(x-y)u_{{\sigma_\varepsilon}^-}(s, x){\rm d}y{\rm d}x{\rm d} s-c^* t\right|<\frac{\varepsilon}{2} t
$$
for all $t \geq T_{\varepsilon}$.

 Fix $\sigma=\sigma_\varepsilon$ in $u_\sigma$ and $u_{\sigma^-}$. Motivated by the arguments in Step~1 and Step~2,
 we can find 
$$
u_{\sigma_\varepsilon}(t, x) \leq \overline u_{\sigma_\varepsilon}(t), \text { for all } t >0
$$
and
$$
 u_{{\sigma_\varepsilon}^-}(t, x) \geq \underline{u}_{\sigma_\varepsilon}(t)-\varepsilon, \text { for all } t >0 $$
locally uniformly for $x$ in $\mathbb R,$
where $\overline u_{\sigma_\varepsilon}(t)$ and $\underline{u}_{\sigma_\varepsilon}(t)$ are the unique time almost periodic solutions of 
$$
{u_t}= u (f(t,  u)+\sigma_\varepsilon)
$$ 
and 
$$
{u}_t={u}(f(t, {u})-\sigma_\varepsilon),
$$
respectively.

Further, according to $(\ref{cc})$, 
for such $\varepsilon$, there is $\hat T>T$ such that $
 \tilde h(t-T)\geq (c^*-\varepsilon)t$ for $t\geq \hat T.$

Denote $$
\mathscr{I}_1(\varepsilon):=(1-\sigma_\varepsilon)^2(1-2\sigma_\varepsilon) \mu\int_{0}^{\underline T_{\sigma_\varepsilon}}\int_{-\infty}^{0}\int_{0}^{\infty}\kappa(x-y)u_{{\sigma_\varepsilon}^-} (s,x){\rm d}y{\rm d}x{\rm d}s-h_{\sigma_\varepsilon}.$$
Take $\underline T_{\sigma_\varepsilon} \gg1$ such that $\mathscr{I}_1(\varepsilon)>0.$
It follows that if
$
 t \geq \max\left\{T_{\sigma_\varepsilon}+\underline T_{\sigma_\varepsilon}+T_\varepsilon+\hat T,  2\dfrac{\mathscr{I}_1(\varepsilon)}{\varepsilon} \right\} \text { and } 0 \leq | x | \leq\left(c^*-\varepsilon\right) t, 
$
then it follows
$$
u(t, x) \leq\left(1-\sigma_\varepsilon\right)^{-2} u_{{\sigma_\varepsilon}}(t, x) \leq\left(1-\sigma_\varepsilon\right)^{-2} \overline {u}_{\sigma_\varepsilon}(t)
$$
and
$$
u(t, x) \geq\left(1-\sigma_\varepsilon\right)^2 u_{{\sigma_\varepsilon}^-}(t, x) \geq\left(1-\sigma_\varepsilon\right)^2\left(\underline{u}_{\sigma_\varepsilon}(t)-\varepsilon\right).
$$
Let $
T^{**}= \max\left\{T_{\sigma_\varepsilon}+\underline T_{\sigma_\varepsilon}+T_\varepsilon+\hat T,  2\dfrac{\mathscr{I}_1(\varepsilon)}{\varepsilon} \right\} ,$
 it follows
$$
\left(1-\sigma_\varepsilon\right)^2\left(\underline{u}_{\sigma_\varepsilon}(t)-\varepsilon\right) \leq u(t, x) \leq\left(1-\sigma_\varepsilon\right)^{-2} \overline{u}_{\sigma_\varepsilon}(t),
$$
for $t \geq T^{**}$ and $0 \leq | x |\leq\left(c^*-\varepsilon\right) t$.
Moreover,  by $(\ref{l0})$ in Step~1, it implies that
$$
\left(1-\sigma_\varepsilon\right)^2\left(\underline{u}_{\sigma_\varepsilon}(t)-\varepsilon\right)-\overline{u}_{\sigma_\varepsilon}(t) \leq u(t, x)-\hat u^*(t) \leq
\left(1-\sigma_\varepsilon\right)^{-2} \overline {u}_{\sigma_\varepsilon}(t)-\underline{u}_{\sigma_\varepsilon}(t)
$$
Let
$$
\mathscr{I}_2(\varepsilon)=\max \left\{\left|\left(1-\sigma_\varepsilon\right)^2\left(\underline{u}_{\sigma_\varepsilon}(t)-\varepsilon\right)-\overline{u}_{\sigma_\varepsilon}(t) \right|,
\left|\left(1-\sigma_\varepsilon\right)^{-2} \overline {u}_{\sigma_\varepsilon}(t)-\underline{u}_{\sigma_\varepsilon}(t)\right|\right\}.
$$
we obtain that
$
\left|u(t, x)-\hat u^*(t)\right| \leq \mathscr{I}_2(\varepsilon)
$
for all $t \geq {T}^{**}$ and $0 \leq |x |\leq\left(c^*-\varepsilon\right) t$.

Since $\mathscr{I}_2(\varepsilon) \rightarrow 0$ as $\varepsilon \rightarrow 0$, it thus yields
$
\lim\limits_{t \rightarrow \infty} \max\limits _{x \leq\left(c^*-\varepsilon\right) t}\left|u(t, x)-\hat u^*(t)\right|=0.
$
\end{proof}

\begin{proof}[The proof of Theorem~\ref{th1-04}]
   Now we turn to prove $(\ref{eq1-8})$. 

Choose a nonnegative, even function sequence $\{\kappa_n\}$ such that  each $\kappa_n(x)\in C^1$ has  nonempty
compact support, and
\begin{equation}\label{eq4-023}
\kappa_n(x)\leq\kappa_{n+1}(x)\leq\kappa(x), \text{ and }  \kappa_n(x)\rightarrow \kappa(x), \text { in } L^1(\mathbb R) \text{ as } n\rightarrow \infty.
\end{equation}
where $\kappa_n(x)= \kappa(x)\chi_n(x)$  and $\{\chi_n\}$  is a properly smooth cut-off function sequences
such that $\kappa_n(x)$  satisfies $\textbf{(H)}$.

Replace $\kappa(x) $ by $\kappa_n(x)$ in $(\ref{eq1-1})$, we can obtain the following auxiliary problem
\begin{equation}\label{eq4-18}
	\left\{\begin{array}{ll}
		u_{ t}=d \int_{g\left(t\right)}^{h\left(t\right)} 
		\kappa_n\left(x-y\right) u\left(t, y\right) \mathrm{d} y-d 
		u+u f(t,u), & t>0,~
		g\left(t\right)<x<h\left(t\right), \vspace{0.1cm}
		 \\ \vspace{0.1cm}
		u\left(t, h(t)\right)=u(t,g(t))=0, & t > 0, \\ \vspace{0.1cm}
		h^{\prime}\left(t\right)= \mu
		\int_{g\left(t\right)}^{h\left(t\right)} 
		\int_{h\left(t\right)}^{\infty} \kappa_n\left(x-y\right) u\left(t, 
		x\right) \mathrm{d} y\mathrm{d} x, 
		& t 
		> 0, \\ \vspace{0.1cm}
		g^{\prime}\left(t\right)=-\mu
		\int_{g\left(t\right)}^{h\left(t\right)} 
		\int_{-\infty}^{g\left(t\right)} \kappa_n\left(x-y\right) u\left(t, 
		x\right) \mathrm{d} y\mathrm{d} 
		x, & 
		t >0, \\ \vspace{0.1cm}
		u\left(0, x\right)=u_{ 0}\left(x\right),~
		h\left(0\right)=-g\left(0\right)=h_{0}, & x\in[-h_{0},h_0].
	\end{array}\right.
\end{equation}

Let $(u_n; g_n,h_n)$ be the solution of $(\ref{eq4-18})$. Applying the similar arguments in Step 2, when the spreading happens, we can see that 
$$
\lim\limits_{t\rightarrow\infty}h_n(t)=\infty, \text{ for } n\gg1$$ and  $(\ref{eq4-15})$ still holds  for $(\ref{eq4-18})$,  then
$$
\begin{aligned}
\liminf_{t\rightarrow\infty} \dfrac{h_n(t)}{t} &\geq \lim_{t\rightarrow\infty}\dfrac{ \mu \int_{0}^{t} \int_{-\infty}^{0 }\int_{0}^{\infty} \kappa_n (x-y)\hat u^*_n(s, x) {\rm d} y{\rm d}x{\rm d}s}{t}
\\&= \lim_{t\rightarrow\infty}\dfrac{ \mu \int_{0}^{t} \int_{-\infty}^{0 }\int_{0}^{\infty} \kappa_n (x-y)\hat u^*_n(s) {\rm d} y{\rm d}x{\rm d}s}{t}
\\&=  \mu\int_{-\infty}^{0 }\int_{0}^{\infty} \kappa_n (x-y)\hat u^*_n  {\rm d} y{\rm d}x
:=\hat c^*_n,
\end{aligned}
$$
where $\hat u^*_n(t, x)\equiv \hat u^*_n(t) $ is the unique positive time  almost periodic solution of 
\[
 u_t=d\int_{\mathbb R}\kappa_n(x-y) u(t,y){\rm d}y -d ~ u(t,x)+uf(t, u),  ~x\in\mathbb R.\]
Since $(\textbf {H})$ does not hold for $\kappa(x)$, by $(\ref{eq4-023})$, it follows that $\lim\limits_{n\rightarrow\infty}\hat c^*_n=\infty.$
Further, by Lemma~4.1 in~\cite{du2021semi}, we can get that for all $n\gg1$, 
$$
 \liminf\limits_{t\rightarrow\infty} \dfrac{h(t)}{t}\geq \hat c^*_n \text { and } \liminf\limits_{t\rightarrow\infty}\dfrac{-g(t)}{t}\geq \hat c^*_n.$$
 Thus, the conclusion $(\ref{eq1-8})$ holds.
The proof of this theorem has been completed.
\end{proof}

}

\section*{Acknowledges} 
This work is supported by the China Postdoctoral Science Foundation (No. 2022M710426) and the National
Natural Science Foundation of China (No. 12171039 and
12271044).

\bibliography{refernon2}

\end{document}